\documentclass[12pt]{amsart}

\usepackage[a4paper]{geometry}

\usepackage{amsfonts}
\usepackage{amsmath}
\usepackage{amsthm}
\usepackage{amssymb}
\usepackage{mathrsfs}
\usepackage{esint}
\usepackage{upgreek}

\usepackage{color}
	\definecolor{red}{rgb}{1,0,0} 
	\definecolor{green}{rgb}{0,1,0} 
	\definecolor{blue}{rgb}{0,0,1} 
	\definecolor{darkblue}{rgb}{0,0,0.6}
	\definecolor{darkred}{rgb}{0.6,0,0}
	
\usepackage[colorlinks,
            linkcolor=blue,
            anchorcolor=darkblue,
            citecolor=darkblue
           ]{hyperref}

\newtheorem{theorem}{Theorem}[section]
\newtheorem{claim}{Claim}
\newtheorem{lemma}[theorem]{Lemma}

\theoremstyle{definition}
\newtheorem{definition}{Definition}[section]

\theoremstyle{remark}

\numberwithin{equation}{section}

\newcommand{\al}{\alpha}
\newcommand{\be}{\beta}
\newcommand{\ga}{\gamma}

\newcommand{\si}{\sigma}

\newcommand{\om}{\omega}

\newcommand{\md}{\mathrm{d}}   
\newcommand{\vd}{\,\md}
\newcommand{\me}{\mathrm{e}}   
\newcommand{\x}{\times}
\newcommand{\pd}{\partial}   

\newcommand{\p}{\ga}  

\newcommand{\R}{\mathbf{R}}   

\newcommand{\PS}{\Omega}  
\newcommand{\E}{\mathbb{E}}  
\newcommand{\Prob}{\mathbb{P}}  
\newcommand{\Filt}{\mathcal{F}}  
\newcommand{\BM}{w}  

\newcommand{\loc}{\mathrm{loc}}

\newcommand{\Dom}{\mathcal{O}}   
\newcommand{\Q}{\mathcal{Q}} 

\newcommand{\pdist}[1]{| #1 |_{{\bf p}}}

\newcommand{\Ct}{C}  
\newcommand{\bL}{L_\om}  

\newcommand{\aij}{a^{ij}}   
\newcommand{\sik}{\si^{ik}}
\newcommand{\sjk}{\si^{jk}}
\newcommand{\bi}{{b^{i}}}

\newcommand{\xdij}{{x^{i}x^{j}}}  
\newcommand{\xdi}{{x^{i}}}

\newcommand{\dij}{{ij}}  
\newcommand{\di}{{i}}

\newcommand{\Dij}{D_{ij}}  
\newcommand{\Di}{D_{i}}

\newcommand{\pc}{{\lambda}}  

\newcommand{\cm}{\varpi}   
\newcommand{\m}{\upkappa} 

\title{
A Schauder estimate for stochastic PDEs
}

\author{Kai Du}
\address
	{Institute for Mathematics and its Applications, School of Mathematics and Applied Statistics,
	University of Wollongong,
	Wollongong, NSW 2522, AUSTRALIA.}
\email{kaid@uow.edu.au}

\author{Jiakun Liu}
\address
	{Institute for Mathematics and its Applications, School of Mathematics and Applied Statistics,
	University of Wollongong,
	Wollongong, NSW 2522, AUSTRALIA.}
\email{jiakunl@uow.edu.au}

\begin{document}

\begin{abstract}
Considering stochastic partial differential equations of parabolic type with random coefficients in vector-valued H\"older spaces, we obtain a sharp Schauder estimate. 
As an application, the existence and uniqueness of solution to the Cauchy problem is also proved.

\end{abstract}
\maketitle

\section{Introduction}

We consider the second-order stochastic partial differential equations (SPDEs) of the It\^o type
\begin{equation}\label{chief}
\md u = (\aij u_\xdij + \bi u_\xdi + c u + f)\vd t + (\sik u_\xdi + \nu^k u + g^k) \vd \BM^k_t,
\end{equation}
in $\R^n\x(0,\infty)$, where $\BM^k$ are countable independent standard Wiener processes
defined on a filtered complete probability space $(\PS,\Filt,(\Filt_t)_{t\in \R},\Prob)$ for $k=1,2,\cdots$.
The matrix $a = (\aij)$ is symmetric, and the uniform {\em parabolic condition} is assumed throughout the paper, namely there is a constant $\pc>0$ such that 
\begin{equation}\label{parab}
2\aij - \sik\sjk \ge \pc \delta_{ij}
\quad\text{on}~~ 
\R^n\x(0,\infty)\x\PS,
\end{equation}
where $\delta_{ij}$ is the Kronecker delta. 
The random fields $u, \aij,\bi,f$ are all real-valued, while $\si^i, \nu$ and $g$ take values in $\ell^2$.
One of the most important examples of \eqref{chief} is the Zakai equation arising in the nonlinear filtering problem \cite{zakai1969}.

The regularity of solutions of \eqref{chief} in Sobolev spaces has already been investigated by many researchers.
Various aspects of $L^2$-theory were studied since 1970s, see \cite{pardoux1975these, krylov1977cauchy, rozovskii1990stochastic, da1992stochastic} and references therein.
Later on, a complete $L^p$-theory was established by Krylov in 1990s, see \cite{krylov1996l_p,krylov1999analytic}. 
By using Sobolev's embedding, one then has the regularity in H\"older spaces, which is however not sharp. 
As an open problem mentioned in \cite{krylov1999analytic}, 
one desires a sharp $C^{2+\alpha}$-theory in the sense that not only that for $f, g$ belonging to a proper space $\mathcal{F}$,
the solution belongs to some kind of stochastic $C^{2+\alpha}$-spaces, 
but also that every element of this stochastic space can be obtained as a solution for certain $f, g$ belonging to the same $\mathcal{F}$.

The purpose of this paper is to establish a Schauder theory of Equation \eqref{chief}, which is sharp in the above sense. 
In order to state our main results, we first introduce a notion of quasi-classical solutions. 

\begin{definition}
A random field $u$ is called a \emph{quasi-classical} solution of \eqref{chief} if
\begin{enumerate}
\item For each $t\in(0,\infty)$, $u(\cdot,t)$ is a twice strongly differentiable function from $\R^n$ to $\bL^{\p} := L^\p(\PS; \R)$ for some $\p\geq2$; and \\
\item For each $x\in\R^n$, the process $u(x,\cdot)$ satisfies \eqref{chief}
in the It\^o integral form with respect to the time variable.
\end{enumerate}
If furthermore, $u(\cdot,t,\om)\in C^2(\R^n)$ for any $(t,\om)\in(0,\infty)\x\PS$, then $u$ is a classical solution of \eqref{chief}.
\end{definition}

Analogously to classical H\"older spaces, we can define the $\bL^\p$-valued H\"older spaces $\Ct^{m+\al}_x(\Q_T;\bL^\p)$ and $\Ct^{m+\al,\al/2}_{x,t}(\Q_T;\bL^\p)$, where $T>0$, $\Q_T=\R^n\x(0,T)$,
and $\bL^{\p} := L^\p(\PS; \R)$ is a Banach space equipped with the norm
$\|\xi\|_{\bL^\p} := (\E |\xi|^{\p})^{1/\p}$.
More specifically, we define $\Ct^{m+\al}_x(\Q_T;\bL^\p)$ to be the set of all $\bL^\p$-valued strongly continuous functions 
$u$ such that 
\begin{equation}\label{norm1}
|u|_{m+\al;\Q_T} := \sup_{\substack{(x,t)\in\Q_T\\ |\be|\le m}} \|D^\be u(x,t)\|_{\bL^\p} 
+ \sup_{\substack{t,\,x\neq y\\ |\be|= m}} 
{\|D^\be u(x,t) - D^\be u(y,t)\|_{\bL^\p} \over |x-y|^\al} < \infty.
\end{equation}
Using the parabolic module
$\pdist{X} := |x| + \sqrt{|t|}$ for $X=(x,t)\in\R^n\x\R$,
we define $C^{m+\al,\al/2}_{x,t}(\Q_T;\bL^\p)$ to be the set of all $u\in C^{m+\al}_x(\Q_T;\bL^\p)$
such that
\begin{equation}\label{norm2}
|u|_{(m+\al,\al/2);\Q_T} := |u|_{m;Q_T} + \sup_{|\be|=m,\ X\neq Y} 
{\|D^\be u(X) - D^\be u(Y)\|_{\bL^\p} \over \pdist{X-Y}^\al} < \infty.
\end{equation}
Similarly, we can define the norms \eqref{norm1} and \eqref{norm2} over a domain $Q=\Dom \x I$, for any domains $\Dom\subset\R^n$ and $I\subset\R$.

\vspace{5pt}
Our main result is the following
\begin{theorem}\label{thm1}
Assume that the classical $C^\al_x$-norms of $\aij, \bi,c,\si^i,\si^i_{\!x},\nu,\nu_x$ are all dominated by a constant $K$
uniformly in $(t,\om)\in(0,T)\x\PS$, and the condition \eqref{parab} is satisfied.
If $f\in \Ct^{\al}_x(\Q_T;\bL^\p)$, $g\in\Ct^{1+\al}_x(\Q_T;\bL^\p)$ for some $\p\ge 2$,
then Equation~\eqref{chief} with a zero initial condition
admits a unique quasi-classical solution $u$ in $\Ct^{2+\al,\al/2}_{x,t}(\Q_T;\bL^\p)$.
\end{theorem}

We remark that the problem with nonzero initial value can be easily reduced to our case by a simple transform.
We also remark that by an anisotropic Kolmogorov continuity theorem (see \cite{dalang2007hitting}), 
if $\al \p > n+2$, the above obtained quasi-classical solution $u$ has a $C^{2+\delta,\delta/2}$ modification for $0<\delta<\al-(n+2)/{\p}$ as a classical solution of \eqref{chief}.

\vspace{5pt}
In order to prove the solvability in Theorem \ref{thm1}, by means of the standard method of continuity, it suffices to establish the following a priori estimate. 

\begin{theorem}\label{schauder}
Under the hypotheses of Theorem \ref{thm1}, 
letting $u \in \Ct^{2,0}_{\loc}(\Q_T;\bL^\p)$ be a quasi-classical solution of \eqref{chief} and $u(\cdot,0) = 0$,
there is a positive constant $C$ depending only on $n,\pc,\p,\al$ and $K$ such that
\begin{equation}\label{schauder.est}
|u|_{(2+\al,\al/2);\Q_T} 
\le C \me^{CT} (|f|_{\al;\Q_T} + |g|_{1+\al;\Q_T}).
\end{equation}
\end{theorem}

The H\"older regularity in spaces $\Ct^{m+\al}_x(\Q_T;\bL^\p)$ for Equation \eqref{chief} was previously investigated by Rozovsky \cite{rozovskiui1975stochastic},
and later was improved by Mikulevicius \cite{mikulevicius2000cauchy}.
However, both works addressed only the equations with nonrandom coefficients and with no derivatives of the unknown function in the stochastic term, namely $\aij$ is deterministic and $\sik\equiv0$.
Moreover, both previous works did not obtain the time-continuity of second-order derivatives of $u$, comparing to our estimate \eqref{schauder.est} and Theorem \ref{thm1}.

The Schauder estimate we obtained in Theorem \ref{schauder} is sharp in the sense that mentioned in \cite{krylov1999analytic}, and is for the general form \eqref{chief} with natural assumptions, where all coefficients are random. 
The approach to $C^{2+\alpha}$-theory in \cite{mikulevicius2000cauchy} was based on several delicate estimates for the heat kernel.
Our method is completely different and more straightforward by combining certain integral estimates and a perturbation argument of Wang \cite{Wang2006}. 
A sketch of proof of Theorem \ref{schauder} is given in Section \ref{s2}.
Full details in addition to applications and further remarks are contained in our separate paper \cite{DL}.

\section{Schauder estimates}\label{s2}

In this section we give an outline of the proof of our main estimate \eqref{schauder.est}.
For simplicity we will first deal with a simplified model equation, and then extend to the general ones.

Consider the model equation
\begin{equation}\label{model}
\md u = (\aij u_\dij + f )\vd t + (\sik u_\di + g^k) \vd \BM^k_t,
\end{equation}
where $\aij, \sik$ are predictable processes, independent of $x$, satisfying the condition~\eqref{parab}. 
We shall consider the model equation in the entire space $\R^n\x\R$.
Suppose that $f(t,\cdot)$ and $g_x(t,\cdot)$ are Dini continuous with respect to $x$ uniformly in $t$, namely
\[
\int_0^1 \frac{\cm(r)}{r} \vd r < \infty,
\]
where
\[
\cm(r) = \sup_{t\in\R,\,|x-y| \le r} 
(\|f(t,x) - f(t,y)\|_{\bL^\p}+\|g_x(t,x) - g_x(t,y)\|_{\bL^\p}).
\]
For any $r>0$, we denote 
\begin{equation}\label{BQ}
B_r(x) = \{y\in\R^n:|y-x|<r\},
\quad
Q_r(x,t) = B_r(x) \x (t-r^2,t),
\end{equation}
and further define $B_r = B_r(0)$ and $Q_r = Q_r(0,0)$.

\begin{lemma}\label{Dini}
Let $u\in C^{2,0}_{x,t}(Q_1;\bL^\p)$ be a quasi-classical solution of \eqref{model}. 
Then there is a positive constant $C$, depending only on
$n,\pc$ and $\p$ such that for any $X,Y \in Q_{1/4}$, 
\begin{equation}\label{Dini.est}
\|u_{xx}(X) - u_{xx}(Y)\|_{\bL^\p}
\le C \left[ \delta M_1
+ \int_0^{\delta} \frac{\cm(r)}{r}\vd r
+ \delta \int_{\delta}^1 \frac{\cm(r)}{r^2}\vd r
\right],
\end{equation}
where $\delta = \pdist{X-Y}$ and 
$ M_1 = |u|_{0;Q_{1}} + |f|_{0;Q_{1}}
+ |g|_{1;Q_{1}}$.
\end{lemma}

An important consequence of Lemma \ref{Dini} is a Schauder estimate that the solution $u\in C^{2+\alpha,\alpha/2}_{x,t}(Q_{1/4};\bL^\p)$ when $f\in C^{\al}_{x}(Q_1;\bL^\p)$ and $g \in C^{1+\al}_{x}(Q_1;\bL^\p)$, for some $\al\in(0,1)$.

\begin{proof}[Outline of proof]
Without loss of generality, we may assume $X=0$.
Let $\rho = 1/2$, and denote
\[
Q^\m = Q_{\rho^\m} = Q_{\rho^\m}(0,0),
\quad
\m = 0,1,2,\cdots.
\]
Construct a sequence of Cauchy problems
\begin{align*}
\md u^\m & = [\aij u^\m_\dij + f(0,t) ]\vd t 
+ [\sik u^\m_\di + g^k(0,t) + g^k_x(0,t) \cdot x] \vd \BM^k_t
\quad \text{in } Q^\m, \\
u^\m & = u \quad \text{on } \pd Q^\m.
\end{align*}

\begin{claim}
For each $\m$, there is a unique generalised solution $u^\m$ such that $u^\m(\cdot,t)\in L^\p(\PS; C^m(B_\varepsilon))$ for any $m\ge0$ and $\varepsilon\in(0,\rho^\m)$. Moreover, for any $r<\rho^\m$ there is a constant $C = C(n,\p)$ such that
\begin{align}\label{Lp}
\|u\|_{L^\p(\PS; L^2(Q_r))}
\le C \bigl( r^{2} \|f\|_{L^\p(\PS; L^2(Q_r))} 
+ r \|g\|_{L^\p(\PS; L^2(Q_r))} \bigr).
\end{align}
\end{claim}

\begin{proof} 
In fact, for $\p=2$, the unique solvability and interior smoothness of $u^\m$ follows from \cite[Theorem~2.1]{krylov1994aw}.
For $\p\ge 2$, higher order $\bL^\p$-integrability \eqref{Lp} can be achieved by a truncation technique.
\end{proof}

\begin{claim}
There is a constant $C = C(n,\pc,\p)$ such that 
\begin{equation}\label{diff}
|D^m(u^\m - u^{\m+1})|_{0;Q^{\m+2}} 
\le C \rho^{(2-m)\m} \cm(\rho^\m), 
\quad 
m=1,2,\dots.
\end{equation}
\end{claim}

\begin{proof}
Note that $(u^\m-u^{\m+1})$ satisfies a homogeneous equation. 
By a delicate computation, we have
\[
|D^m (u^\m - u^{\m+1})|_{0;Q^{\m+2}}
\le C \rho^{-m \m} 
\biggl{\|} \fint_{Q^{\m+1}} (u^\m - u^{\m+1})^2 \vd X \biggr{\|}_{\bL^{\p/2}}^{1/2}
=: I_{\m,m}.
\]
On the other hand, $(u^\m-u)$ satisfies a zero initial condition. 
By Claim 1, 
\[
J_\m := \biggl{\|} \fint_{Q^{\m}} (u^\m - u)^2 \vd X \biggr{\|}_{\bL^{\p/2}}^{1/2}
\le C \rho^{2\m} \cm(\rho^\m).
\]
Thus, Claim 2 is proved, since
\[
I_{\m,m} \le C \rho^{-m\m} (J_\m + J_{\m+1}) \le C \rho^{(2-m)\m}\cm(\rho^\m).
\]
It is worth remarking that instead of using the maximum principle to estimate the term $|D^m(u^\m - u^{\m+1})|_{0;Q^{\m+2}}$ as in \cite{Wang2006}, we obtain the inequality \eqref{diff} by subtle integral estimates. 
\end{proof}

\begin{claim}
$\{u^\m_{xx}(0)\}$ converges in $\bL^\p$ (here $0\in\R^{n+1}$), and the limit is $u_{xx}(0)$.
\end{claim}

\begin{proof}
By Claim 2 and the assumption of Dini continuity, 
\[
\sum_{\m\ge 1} |(u^\m - u^{\m+1})_{xx}|_{0;Q^{\m+2}} 
\le C \sum_{\m\ge 1} \cm(\rho^\m) 
\le C \int_0^1 {\cm(r) \over r} \vd r
< \infty,
\]
which implies that $u^\m_{xx}(0)$ converges in $\bL^\p$.
Since $\p\ge 2$, it suffices to show that
\begin{equation}\label{conv}
\lim_{\m\to\infty}\|u^\m_{xx}(0) - u_{xx}(0)\|_{\bL^2}
=0,
\end{equation}
which can also be achieved straightforward by our integral estimates. 
\end{proof}

Now for any $Y=(y,s)\in Q_{1/4}$ we can select an $\m$ such that  
$\pdist{Y} \in [\rho^{\m+2}, \rho^{\m+1})$.
By decomposition, one has
\begin{align}  \label{main est}
& \|u_{xx}(Y) - u_{xx}(0)\|_{\bL^\p} \nonumber \\
& \le \|u^\m_{xx} (Y) - u^\m_{xx} (0)\|_{\bL^\p}
+ \|u^\m_{xx} (0) - u_{xx} (0)\|_{\bL^\p}
+ \|u^\m_{xx} (Y) - u_{xx}(Y)\|_{\bL^\p} \\
& =: I_1 + I_2 + I_3. \nonumber
\end{align}

\begin{claim}
$I_1 \leq C \delta M_1 + C  \delta
\int_{\delta}^1  {\cm(r) \over r^2} \vd r$, where $\delta:=\pdist{Y}$ and $M_1$ was given in \eqref{Dini.est}.
\end{claim}

\begin{proof}
The proof is by induction. 
When $\m=0$, note that $u_{xx}^0$ satisfies the following homogeneous equation:
\begin{align*}
\md u^0_{xx} & = \aij \Dij u^0_{xx} \vd t + \sik \Di u^0_{xx} \vd \BM^k_t
\quad \text{in } Q_{3/4}.
\end{align*}
From interior estimates, we have
\begin{equation}\label{u0est}
\|u_{xx}^0(X) - u_{xx}^0(Y)\|_{\bL^\p}
\le C M_1 \pdist{X-Y},
\quad\forall\,X,Y\in Q_{1/4}.
\end{equation}
When $\m \ge 1$, denote $h^\iota = u^\iota - u^{\iota-1}$, for $\iota = 1,2,\dots,\m$,
then $h^\iota$ satisfies 
\begin{align*}
\md h^\iota & = \aij h^\iota_\dij \vd t + \sik h^\iota_\di \vd \BM^k_t
\quad \text{in } Q^{\iota}.
\end{align*}
From Claim 2, we have for $-\rho^{2(\m+1)} \le t \le 0$ and $|x| \le \rho^{\m+1}$,
\begin{equation}\label{u0-eq}
\|h^\iota_{xx}(x,t) - h^\iota_{xx}(0,0)\|_{\bL^\p}
\le  C \rho^{\m-\iota}\cm(\rho^{\iota-1}).
\end{equation}
Using \eqref{u0est} and \eqref{u0-eq}, we can obtain the estimate 
\begin{equation*}\label{ukest}
\begin{split}
I_1
& \le \|u^{\m-1}_{xx}(Y) - u^{\m-1}_{xx}(0)\|_{\bL^\p}
+ \|h^\m_{xx}(Y) - h^\m_{xx}(0)\|_{\bL^\p} \\
& \le \|u^{0}_{xx}(Y) - u^{0}_{xx}(0)\|_{\bL^\p}
+ \sum_{\iota=1}^\m \|h^\iota_{xx}(Y) - h^\iota_{xx}(0)\|_{\bL^\p} \\
& \le C \delta M_1 + C  \delta
\int_{\delta}^1  {\cm(r) \over r^2} \vd r.
\end{split}
\end{equation*}
Claim 4 is proved.
\end{proof}

\begin{claim} 
$I_i \leq C \int_0^{\delta} \! {\cm(r) \over r} \vd r$, for $i=2,3$.
\end{claim}

\begin{proof}
The estimate of $I_2$ is a refinement of {convergence} in Claim 3. 
In fact, by Claim 2 we have the precise estimate 
\begin{equation}\label{conv.rate}
I_2=\|u^\m_{xx}(0) - u_{xx}(0)\|_{\bL^\p}
\le \sum_{j\ge \m} |(u^j - u^{j+1})_{xx}|_{0;Q^{j+2}} 
\le C \int_0^{\rho^\m} \! {\cm(r) \over r} \vd r,
\end{equation}
where $C = C(n,\pc,\p)$.
We can obtain a similar estimate for $I_3$ by shifting the centre of domains. 
\end{proof}

To sum up, Lemma \ref{Dini} is proved.
\end{proof}

Having proved Lemma \ref{Dini} we are in a position to derive the global estimate of solutions of \eqref{chief} and complete the proof of Theorem \ref{schauder}.

\begin{proof}[Outline of proof of Theorem \ref{schauder}]
The proof is by an argument of frozen coefficients. 
Denote $\Q_{r,\tau} = B_r \x (0,\tau)$, and let
\[
{M}^\tau_{x,r}(u) 
= \sup_{0\le t \le \tau}
\biggl(\fint_{B_r(x)} \!\E\,|u(t,y)|^\p \vd y \biggr)^{1/\p}
\quad
{M}^\tau_{r}(u) = \sup_{x\in\R^n}{M}^\tau_{x,r}(u).
\]
By multiplying cut-off functions and applying Lemma \ref{Dini} we can get
\begin{align}\label{pf404}
|u_{xx}|_{(\al,\al/2);\Q_{\rho/2,\tau}}
\le C \Bigl(  {M}^\tau_{0,\rho}(u)
+ |f|_{\al;\Q_{\rho,\tau}} 
+ |g|_{1+\al;\Q_{\rho,\tau}} \Bigr),
\end{align}
for some sufficiently small $\rho>0$.  
The derivation of \eqref{pf404} involves a rather delicate computation, which makes use of interpolation inequalities in H\"older spaces 
(see \cite[Lemma~6.35]{gilbarg2001elliptic} or \cite[Theorem~3.2.1]{krylov1996lectures}).
Since the centre of domains can shift to any point $x \in \R^n$, we obtain 
\begin{align}\label{pf403}
|u|_{(2+\al,\al/2);\Q_{\tau}}
\le C \Bigl( {M}^\tau_{\rho}(u)
+ |f|_{\al;\Q_{\tau}} 
+ |g|_{1+\al;\Q_{\tau}} \Bigr),
\end{align}
where $C = C(n,\pc,\p,\al)$.

To estimate ${M}_\rho^{\tau}(u)$, applying It\^o's formula, and using H\"older and Sobolev-Gagliargo-Nirenberg inequalities, we can get
\begin{align*}
{M}^\tau_{\rho}(u) \le
C_1 \tau ({M}^\tau_{\rho}(u) + |u_{xx}|_{0;Q_{\rho,\tau}} + |f|_{0;\Q_{\tau}} 
+ |g|_{0;\Q_{\tau}}),
\end{align*}
where $C_1 = C_1(n,\pc,\p)$.
Letting $\tau = (2CC_1 + C_1)^{-1}$,
by virtue of \eqref{pf403} we obtain 
\begin{align}\label{tau}
|u|_{(2+\al,\al/2);\Q_{\tau}}
\le C_0 \bigl( |f|_{\al;\Q_{\tau}} 
+ |g|_{1+\al;\Q_{\tau}} \bigr),
\end{align}
where $C_0 = C_0(n,\pc,\p,\al)$.

Finally, the proof of \eqref{schauder.est} and Theorem \ref{schauder} is completed by induction.
\end{proof}

\bibliographystyle{amsalpha}
\small{
\bibliography{ref_spde}
}

\end{document}